\documentclass[12pt]{article}%
\usepackage{amsmath}
\usepackage{amsfonts}%
\usepackage{amssymb}%
\usepackage{graphicx}

\providecommand{\U}[1]{\protect\rule{.1in}{.1in}}
\newtheorem{theorem}{Theorem}

\newtheorem{claim}[theorem]{Claim}

\newtheorem{corollary}[theorem]{Corollary}

\newtheorem{lemma}[theorem]{Lemma}

\newtheorem{proposition}[theorem]{Proposition}

\newenvironment{proof}[1][Proof]{\noindent\textbf{#1.} }{\ \rule{0.5em}{0.5em}}
\newcommand{\D}{D}

\title{Bessel Models for General Admissible Induced Representations: The Compact Stabilizer Case}
\begin{document}

\maketitle

\begin{abstract}
A holomorphic continuation of Jacquet type integrals for parabolic subgroups
with abelian nilradical is studied. Complete results are given for generic
characters with compact stabilizer and arbitrary representations induced from
admissible representations. A description of all of the pertinent examples is
given. These results give a complete description of the Bessel models
corresponding to compact stabilizer.

\end{abstract}
\section{Introduction.}

Classically an automorphic form on the upper half plane, $\mathcal{H}$, of
weight $k$ is a holomorphic function, $f$, satisfying a condition at infinity
(described below), on $\mathcal{H}$ such that
\[
f\left(  \frac{az+b}{cz+d}\right)  =(cz+d)^{k}f(z)
\]
for
\[
\left[
\begin{array}
[c]{cc}%
a & b\\
c & d
\end{array}
\right]  \in\Gamma
\]
with $\Gamma$ a subgroup of $SL(2,\mathbb{Z})$ of finite index. In particular, if
\[
N=\left\{  \left[
\begin{array}
[c]{cc}%
1 & x\\
0 & 1
\end{array}
\right]  |x\in\mathbb{R}\right\}
\]
then $N\cap\Gamma$ is the infinite cyclic group generated by
\[
\left[
\begin{array}
[c]{cc}%
1 & h\\
0 & 1
\end{array}
\right]
\]
with $h>0$ (called the \emph{height }of the cusp). This implies that as a
function of $z$, $f$ is periodic of period $h$. Thus if we set
\[
\tau=e^{\frac{2\pi iz}{h}}%
\]
then we can write
\[
f=\sum a_{k}\tau^{k}.
\]
The condition at infinity alluded to above implies that $a_{k}=0$ for all
$k<0$. Of particular importance are the cusp forms ($a_{k}=0$ if $k\leq0$).
Associated with such a form is the Dirichlet series (in the case when the
height of the cusp is $1$)%
\[
\sum_{n=1}^{\infty}\frac{a_{n}}{n^{z}}.
\]
These are the automorphic L-functions that have played such an important role
in the recent advances in number theory. These numbers $a_{n}$ have a
beautiful representation theoretic interpretation. The function $f$ can be
lifted to an analytic function, $\phi$, on $SL(2,\mathbb{R})$ that satisfies
\[
\phi(\gamma g)=\phi(g),\gamma\in\Gamma
\]
and if
\[
k(\theta)=\left[
\begin{array}
[c]{cc}%
\cos\theta & \sin\theta\\
-\sin\theta & \cos\theta
\end{array}
\right]  \in SO(2)
\]
then%
\[
\phi(gk(\theta))=\phi(g)e^{ik\theta}.
\]
The holomorphy condition implies that the representation of $\mathfrak{g}%
=Lie(SL(2,%
\mathbb{R}))$ on the span of the right regular derivatives of $\phi$ by the enveloping
algebra of $\mathfrak{g}$ defines an irreducible representation of
$\mathfrak{g}$ equivalent to the space of $SO(2)$-finite vectors in an
irreducible unitary representation of $SL(2,
\mathbb{R})$, $(\pi,H)$, which is an element of the holomorphic discrete series if
$k>1$. The interpretation is that if $H^{\infty}$ is the Fr\'{e}chet space of
$C^{\infty}$ vectors in $H$ then there is for each $n$ a continuous functional
on $H^{\infty}$, $W_n$, such that $W_n(f)=a_{n}$ and has the property that%
\[
W_n(\pi\left(  \left[
\begin{array}
[c]{cc}%
1 & x\\
0 & 1
\end{array}
\right]  v\right)  =e^{2\pi inx}W_n(v)
\]
for all $x$ in $\mathbb{R}$. The theory of Maass extends the class of automorphic forms to include all
types of unitary representations of $SL(2,\mathbb{R})$. The analogous functionals, $W$, are, following Jacquet, called Whittaker
functionals corresponding to the unitary character $\chi_{n}$ of $N$ given by
\[
\chi_{n}\left(  \left[
\begin{array}
[c]{cc}
1 & x\\
0 & 1
\end{array}
\right]  \right)  =e^{2\pi inx},
\]

If, we move from the upper half plane to the Siegel upper half plane,
$\mathcal{H}_{n}$, consisting of elements $Z=X+iY$ with $X$ and $Y$ symmetric $n\times n$ matrices over $\mathbb{R}$ and $Y$ positive definite. The symplectic group $G=Sp(n,
\mathbb{R})$ consisting of real $2n\times2n$ matrices with block form
\[
g=\left[
\begin{array}
[c]{cc}
A & B\\
C & D
\end{array}
\right]
\]
with $A,B,C,D$, $n\times n$ such that if
\[
J=\left[
\begin{array}
[c]{cc}%
0 & I\\
-I & 0
\end{array}
\right]  ,
\]
with $I$ the $n\times n$ identity matrix, then%
\[
gJg^{T}=J.
\]
Then $G$ acts on $\mathcal{H}_{n}$ by
\[
gZ=(AZ+B)(CZ+D)^{-1}.
\]
In this case C.L.Siegel [S] considered subgroups $\Gamma$ of finite index in
$G_{
\mathbb{Z}}=Sp(n,\mathbb{Z})$ and holomorphic functions, $f$, on $\mathcal{H}_{n}$ such that (with $g$ in
block form as above)
\[
f(gZ)=\det\left(  CZ+D\right)  ^{k}f(Z)
\]
for $g\in\Gamma$ and a growth condition at $\infty$. (In later work the factor
$\det\left(  CZ+D\right)  ^{k}$ was replaced by $\sigma(CZ+D)$ with $\sigma$ a
finite dimensional irreducible representation of $GL(n,
\mathbb{C})$) As above one has the subgroup $N$ consisting of the elements of the form%
\[
n(L)=\left[
\begin{array}
[c]{cc}%
I & L\\
0 & I
\end{array}
\right]
\]
with $L$ an $n\times n$ symmetric matrix over $
\mathbb{R}$ and $\Gamma\cap N$ contains a subgroup of finite index in $N_{\mathbb{Z}}=G_{\mathbb{Z}
}\cap N$. We will assume (as above) that $\Gamma\cap N=N_{
\mathbb{Z}
}$. Thus,
\[
f(Z+L)=f(Z)
\]
for $L$ an $n\times n$ symmetric matrix with entries in $\mathbb{Z}$. 
We can thus expand this Siegel modular form in a Fourier series.%
\[
\sum a_{S}e^{2\pi iTr(SZ)}%
\]
the sum over $S$ a symmetric $n\times n$ matrix over $\mathbb{Z}$. 
One finds that if $a_{S}\neq0$ then $S$ must be positive semi-definite. As
above these coefficients have a beautiful representation theoretic
interpretation (the one that will be expanded on in this paper). We can
consider
\[
\chi_{S}(n(X))=e^{2\pi iTr(SX)}%
\]
with $S$ symmetric $n\times n$ matrix over $\mathbb{R}$. 
We have $GL(n,\mathbb{R})$ imbedded in $G$ via
\[
g\longmapsto\left[
\begin{array}
[c]{cc}
g & 0\\
0 & g^{T}
\end{array}
\right]
\]
the image subgroup in $G$ will be denoted $M$ and $P=MN$ is the Siegel
parabolic subgroup of $G$. $M$ acts on $N$ by conjugation and hence acts on
its unitary characters. One finds that if a character is non-degenerate (that
is the $M$-orbit in the character group is open) then the character must be
given by $\chi_{S}$ with $\det S\neq0$. The stabilizer of the character is
compact only if $S$ is positive or negative definite. One can show (c.f. W[5])
that the only such models for holomorphic (resp. antiholomorphic)
representations correspond are those corresponding to positive definite (resp.
negative definite) such $S$. Thus the only generic characters that can appear
if we consider holomorphic or anti-holomorphic Siegel modular forms are the
ones with compact stabilizer. As in the case of $SL(2,\mathbb{R})$
 ($n=1$) in number theoretic applications it is necessary to study analogues
of the distributions $W$ as above on the space of smooth vectors of a general
irreducible unitary representation for the characters $\chi_{S}$ with compact stabilizer.

If $G$ is a simple group over $\mathbb{R}$ 
then we will call a subgroup with abelian unipotent radical,
\textquotedblleft Bessel type\textquotedblright\ if the unipotent radical
admits a unitary character with compact stabilizer in a Levi factor of the
parabolic subgroup. In this paper we consider spaces of functionals, $W$, on
admissible irreducible representations that admit covariants (as above) by
unitary characters with compact stabilizer (the Bessel functionals). These are
the Bessel models in our title). Our main result yields a complete
determination of Bessel modules for admissible representations induced from
Bessel parabolic subgroups. Using this theorem we deduce (using results in
section 3) an upper bounds on the dimensions of the spaces of Bessel
functionals on irreducible smooth Fr\'{e}chet representations of moderate
growth. This result is an inequality relating Bessel models to spaces of
conical vectors.

In the second section of the paper we give a complete list of simple groups
having a non-minimal parabolic subgroup with abelian nil-radical and a
character with compact stabilizer in a Levi factor. As it turns out, the class
consists of the groups whose symmetric spaces are the irreducible Hermitian
symmetric spaces of tube type of rank greater than 1. We describe the
differentials of desired characters in the next section. We feel that this
discussion is of independent interest and except for the cases corresponing to
$SO(n,2)$ (tube domains over the light cone) there is one classical family
corresponding the each skew field over $\mathbb{R}$ ($Sp_{2n}(\mathbb{R})$ 
for $\mathbb{R}$, $SU(n.n)$
 for $\mathbb{C}$
,$SO^{\ast}(4n)$ for the quaternions) and the classical simple Euclidian Jordan
algebras (the Hermitain $n\times n$ matrices over the field). In addition, the
exceptional case of $E_{7}$ is gotten from the exceptional Jordan algebra,
that is, the $3\times3$ Hermitian matrices over the octonians. The
corresponding compact stabilizers are the unitary groups over the
corresponding fields ($O(n),U(n),Sp(n)$ and compact $F_{4}$). This material is
certainly not new, however, the uniform interpretation of the unitary
characters with compact stabilizer is important to our later developments and
is beautiful in its own right. the proof that these are the only examples
would take us afield so we have therefore decided against including it in this paper.

Another novel feature of this paper is the purely algebraic material in
section 3 of this paper. Which substantially simplifies the results of [W2]
for the case of parabolic subgroups with abelian nilradical. In particular, in
this case following the argument of [W2] one can also give a simple proof of
Lynch's vanishing theorem. We do include a proof the exactness of the related
spaces of covariants directly from the results in this section.

We will now say a few words about the details main results in this paper. Let
$G$ be a real reductive group with maximal compact subgroup $K$. Let $P=MAN$
be a parabolic subgroup with given Langlands decomposition. Let $(\sigma
,H_{\sigma})$ be an admissible, finitely generated, smooth Fr\'{e}chet
representation of moderate growth of $M$ and let $\nu$ be a complex valued
linear functional on $Lie(A)$. Let $I_{P,\sigma,\nu}^{\infty}$ be the smooth
induced representation of $G$ from the representation, $\sigma_{\nu}$, of $P$
given by $\sigma_{\nu}(man)=e^{\nu(\log a)}\sigma(m)$. Let $\overline
{P}=MA\overline{N}$ be the opposite parabolic subgroup (corresponding to the
given Langlands decomposition). Let $\chi$ be a unitary (one dimensional
character) of $\overline{N}$. We consider the integrals%
\[
J_{P,\sigma,\nu}^{\chi}(f)=\int_{\overline{N}}\chi(\overline{n})^{-1}%
f_{P,\sigma,\nu}(\overline{n})d\overline{n}%
\]
where $d\overline{n}$ is a choice of Haar measure on $\overline{N}$ and
\[
f_{P,\sigma,\nu}(namk)=e^{\nu(\log a)}\delta_{P}(a)^{\frac{1}{2}}%
\sigma(m)f(k)
\]
for $f$ in the representation induced by $\sigma_{|K\cap M}$ from $K\cap M$ to
$K,I_{M\cap K,\sigma_{|M\cap K}}^{\infty}$, and $\delta_{P}$ is the modular
function of $P$. These integrals converge uniformly on compacta for
$\operatorname{Re}\nu$ in a suitable translate of the positive cone. In
earlier work the second named author and various other researchers have
studied the analytic continuation of these integrals to all of
$Lie(A)_{\mathbb{C}}^{\ast}$ in the case when $\chi$ is generic (i.e. the
$M$-orbit is open) in the case when $\sigma$ is finite dimensional and $P$
satisfies appropriate conditions. The most general results in this direction
can be found in [W1] for earlier work the reader should consult the references
therein. In this paper we embark on the problem of implementing this
continuation in the case when $H_{\sigma}$ is infinite dimensional. One allow
this condition to handle all admissible irreducible smooth Fr\'{e}chet modules
of moderate growth and also to handle applications to number theory involving
general representations. Here the situation becomes much more complicated
analytically. We will confine ourselves to a special case of what we call a
Bessel model. That is, $N$ is commutative and the stabilizer of $\chi$ in $M$,
$M_{\chi}$, is compact. One can show that in this case $M_{\chi}$ is a
symmetric subgroup of $M$. Also, in this case, there is an element, $w_{M}$,
in the normalizer of $A$ in $K$ that conjugates $P$ to $\overline{P}$. We will
thus consider a generic character, $\chi$, of $N$ and rewrite the integral
above as%
\[
J_{P,\sigma,\nu}^{\chi}(f)=\int_{N}\chi(n)^{-1}f_{P,\sigma,\nu}(w_{M}n)dn.
\]
We prove a holomorphic continuation for these integrals . We also prove that every Bessel model of $I_{P,\sigma,\nu
}^{\infty}$ factors through these integrals. If we fix a $M_{\chi}$ type then
these results allow us to calculate the dimension of the space of such models.
For example, if a $M_{\chi}$ type has multiplicity $1$ in $\sigma$ then the
dimension of the corresponding space is $1$. In section 6 we derive upper
bounds on the dimensions of the spaces transforming by a given representation
of $M_{\chi}$.

There are several new difficulties that arise in the more general case of
noncompact stabilizer including the fact that there is more than one open
double coset. These questions will be addressed in forthcoming work of the
first named author. We should also point out that the aspect of our work that
implies multiplicity one has overlap with work of Jiang,Sun and Zhu [JSZ] in
the special case of $SO(n,2)$

We would like to thank Dipendra Prasad for initiating this research by asking
whether the results as described are true for $GSp_{4}(\mathbb{R})$. Since its
commutator group is the rank 2 group corresponding to $\mathbb{R}$ below, this
paper answers his question in the affirmative.

\section{A class of examples.} \label{sec:examples}

In this section we will describe a collection of real reductive groups for
which the results of this paper apply. Let $G$ be a connected simple Lie group
with finite center and let $K$ be a maximal compact subgroup. We assume that
$G/K$ is Hermitian symmetric of tube type. This can be interpreted as follows.
There exists a group homomorphism, $\phi$, of a finite covering, $S$ of $PSL(2,\mathbb{R})$
into $G$ such that if $H=\phi(S)$
then $H\cap K$ is the center of $K$. We take a standard basis $h,e,f$ of
$Lie(H)$ over
$\mathbb{R}$
 with the standard TDS (three dimensional simple) commutation relations
($[e,f]=h,[h,e]=2e,[h,f]=-2f$). If $\mathfrak{g}=Lie(G)$ we have
$\mathfrak{g}=\mathfrak{\bar{n}\oplus m\oplus n}$ with $\mathfrak{\bar{n}%
,m,n}$ respectively the $-2,0,2$ eigenspace of $ad(h).$ In particular,
$\mathfrak{\bar{n}}$ and $\mathfrak{n}$ are commutative and $e\in\mathfrak{n}%
$, $f\in\mathfrak{\bar{n}}$. Let $\theta$ be the Cartan involution of $G$
corresponding to the choice of $K$ then we may assume, $\theta\mathfrak{n=\bar
{n}}$, $f=-\theta e$ and
$\mathbb{R}(e-f)=Lie(H\cap K)$. Set $\mathfrak{p=}$ $\mathfrak{m\oplus n}$ and let
$P=\{g\in G|Ad(g)\mathfrak{p}=\mathfrak{p\}}$. Then $P$ is a parabolic
subgroup of $G$. Let $M=\{g\in G|Ad(g)h=h\}$. Then $M$ is a Levi factor of $P
$ and $N=\exp(\mathfrak{n})$ is its unipotent radical. With this notation in
place we can describe some remarkable properties of these spaces.

\begin{lemma}
Let $M^{f}=\{m\in M|Ad(m)f=f\}$. Then $M^{f}=M\cap K$.
\end{lemma}

\begin{proof}
We note that as a representation of $H$ under $ad$, $\mathfrak{g}\cong
aF^{0}\oplus bF^{2}$. Here $F^{k}$ is an irreducible representation of
$Lie(H)$ of dimension $k+1$ (i.e. highest weight $k$). Here $b=\dim
\mathfrak{n}$ and $a=\dim\mathfrak{g}-3b$. On the other hand $\mathfrak{m}$ is
the zero weight space for the action of $h.$ Furthermore, using standard
representation theory of a TDS we see that if $F$ is a TDS module and if $v\in
F$ is such that $hv=0$ and $fv=0$ then $ev=0$. this implies that $a=\dim
M^{f}$. We also note that the same argument implies that $adf:\mathfrak{n}%
\rightarrow\mathfrak{m}$ is injective. Since $Ad(m)f=f$ and $Ad(m)h=h$ for
$m\in M^{f}$ we see that%
\[
\lbrack Ad(m)e,f]=[Ad(m)e,Ad(m)f]=Ad(m)h=h
\]
so%
\[
\lbrack f,e-Ad(m)e]=0.
\]
Hence $Ad(m)e=e$ for $m\in M^{f}$. This implies that $M^{f}\subset K=\{g\in
G|Ad(g)(e-f)=(e-f)\}$. We note that $M\cap K$ is the group of all $m\in G$
such that $Ad(g)h=h$ and $Ad(g)(e-f)=e-f$ (and hence $Ad(g)[h,e-f]=[h,e-f]$ ).
Thus $M\cap K\subset M^{f}$.
\end{proof}

We now give a complete listing of the examples:

1. $G=Sp(n,\mathbb{R})$ realized as $2n\times2n$ matrices such that $gJ_{n}g^{T}=J_{n}$ with%
\[
J_{n}=\left[
\begin{array}
[c]{cc}%
0 & I_{n}\\
-I_{n} & 0
\end{array}
\right]
\]
with $I_{n}$ the $n\times n$ identity matrix (upper $T$ means transpose).
$\theta(g)=(g^{-1})^{T}$. Thus identifying $\mathbb{R}^{2n}$ with
$\mathbb{C}$
using $J_{n}$ for the complex structure, $K=U(n)=G\cap O(2n)$. Here
\[
e=\left[
\begin{array}
[c]{cc}%
0 & I_{n}\\
0 & 0
\end{array}
\right], f=\left[
\begin{array}
[c]{cc}%
0 & 0\\
I_{n} & 0
\end{array}
\right], h=\left[
\begin{array}
[c]{cc}%
I_{n} & 0\\
0 & -I_{n}%
\end{array}
\right]  .
\]
Thus
\[
M=\left\{  \left[
\begin{array}
[c]{cc}%
g & 0\\
0 & (g^{-1})^{T}%
\end{array}
\right]  |g\in GL(n,%
\mathbb{R}
)\right\}  ,
\]%
\[
N=\left\{  \left[
\begin{array}
[c]{cc}%
I & X\\
0 & I
\end{array}
\right]  |X\in M(n,%
\mathbb{R}
),X^{T}=X\right\}  .
\]
Finally, $M^{f}=O(n)\subset U(n)$.

2. $G=SU(n,n)$ realized as the $2n\times2n$ complex matrices, $g$, such that
$gL_{n}g^{\ast}=L_{n}$ with%
\[
L_{n}=\left[
\begin{array}
[c]{cc}%
0 & iI_{n}\\
-iI_{n} & 0
\end{array}
\right]  .
\]
Then $K=U(2n)\cap G\cong S(U(n)\times U(n))$ the center of $K$ is the set of
all
\[
\left[
\begin{array}
[c]{cc}%
\left(  \cos\theta\right)  I_{n} & \left(  \sin\theta\right)  I_{n}\\
-\left(  \sin\theta\right)  I_{n} & \left(  \cos\theta\right)  I_{n}%
\end{array}
\right]
\]
with $\theta\in%
\mathbb{R}
$. We take the same $e,f,h$ as in the case of $Sp(n,%
\mathbb{R}
)$. The centralizer, $M$, of $h$ in $G$ is the set of all%
\[
\left[
\begin{array}
[c]{cc}%
g & 0\\
0 & (g^{\ast})^{-1}%
\end{array}
\right]
\]
with $g\in GL(n,%
\mathbb{C}
).$ $M^{f}$ is the set of elements in $M$ with $g\in U(n).$

3. $G=SO^{\ast}(4n)$ realized as the group of all $g\in SO(4n,%
\mathbb{C}
)$ such that $gJ_{2n}g^{\ast}=J_{2n}$. Here $K$ is isomorphic with $U(2n)$.
Here $K=G\cap S0(4n,%
\mathbb{R}
)=Sp(2n,%
\mathbb{R}
)\cap S0(4n,%
\mathbb{R}
)\cong U(2n)$. \ We can describe $\mathfrak{g}=Lie(G)$ as a Lie subalgebra of
$M_{2n}(\mathbb{H})$ as the matrices in block form%
\[
\left[
\begin{array}
[c]{cc}%
A & X\\
Y & -A^{\ast}%
\end{array}
\right]
\]
with $A,X,Y\in M_{n}(\mathbb{H})$ and $X^{\ast}=X,Y^{\ast}=Y$. In this form
$\mathfrak{g}\cap M_{2n}(\mathbb{R})=Lie(Sp(n,%
\mathbb{R}
))$. \ We take $e,f,h$ as above and note that $M\cong GL(n,\mathbb{H})$ and
$M^{f}=Sp(n)$ the quaternionic unitary group.

4. $G$ the Hermitian symmetric real form of $E_{7}$. In this case we will
emphasize a decomposition of $Lie(G)$ which makes it look exactly like those
examples 1.,2., and 3.. In each of those cases we have
\[
Lie(G)=\left[
\begin{array}
[c]{cc}%
A & X\\
Y & -A^{\ast}%
\end{array}
\right]
\]
with $A$ an element of $M_{n}(F)$ and $F=%
\mathbb{R}
,%
\mathbb{C}
$ or $\mathbb{H}$ the upper * is the conjugate (of the field) transposed.
Furthermore, $X,Y$ are elements of $M_{n}(F)$ that are self adjoint. Example 5
corresponds to the octonions, $\mathbb{O}$. Here we replace $M_{3}%
(\mathbb{O)}$ by $\mathfrak{m}=%
\mathbb{R}
\oplus E_{6,2}$ (the real form of real rank 2 with maximal compact of type
$F_{4}$). We take for $X,Y$ elements of the exceptional Euclidean Jordan
algebra (the $3\times3$ conjugate adjoint matrices over $\mathbb{O}$ with
multiplication $A\cdot B=\frac{1}{2}(AB+BA)$ thus in this case the $X$'s and
$Y$'s are defined in the same way for the octonions as for the other fields).
Here $\mathfrak{m}$ acts by operators that are a sum of Jordan multiplication
and a derivation of the Jordan algebra (the derivations defining the Lie
algebra of compact $F_{4}$). With this notation our choice of $e,f,h$ are
exactly the same as the examples for $%
\mathbb{R}
,%
\mathbb{C}
$ or $\mathbb{H}$.

There is one more example (that doesn't fit this beautiful picture).

5. $G=SO(n,2)_{o}$ realized as the identity component of the group $n+2$ by
$n+2$ matrices that leave invariant the form
\[
\left[
\begin{array}
[c]{cc}%
I_{n} & 0\\
0 & -I_{2}%
\end{array}
\right]  .
\]
Here
\[
h=\left[
\begin{array}
[c]{cccc}%
0_{n-1} & 0 & 0 & 0\\
0 & 0 & 2 & 0\\
0 & 2 & 0 & 0\\
0 & 0 & 0 & 0
\end{array}
\right]
\]
where the $0$'s on the edges are either $n-1\times1$ or $1\times n-1$. Also%
\[
e=\left[
\begin{array}
[c]{cccc}%
0_{n-1} & 0 & 0 & 0\\
0 & 0 & 0 & 1\\
0 & 0 & 0 & 1\\
0 & 1 & -1 & 0
\end{array}
\right]  ,f=\left[
\begin{array}
[c]{cccc}%
0_{n-1} & 0 & 0 & 0\\
0 & 0 & 0 & 1\\
0 & 0 & 0 & -1\\
0 & 1 & 1 & 0
\end{array}
\right]  .
\]
Here
\[
K=\left[
\begin{array}
[c]{cc}%
k_{1} & 0\\
0 & k_{2}%
\end{array}
\right]
\]
with $k_{1}\in SO(n)$ and $k_{2}=SO(2).$ Thus $e-f$ is an infinitesimal
generator of the center of $K$ (if $n>2$). $M$ is isomorphic with
\[%
\mathbb{R}
^{\times}\times_{\mu_{2}}O(n-1,1)
\]
and
\[
M^{f}\cong O(n-1).
\]

\section{The tensoring with finite dimensional representations.}\label{sec:tensoring}

Let $\mathfrak{g}$ be a reductive Lie algebra over $%
\mathbb{C}
$. Choose an invariant symmetric, nondegenerate, bilinear form, $(...,...)$,
on $\mathfrak{g.}$ Let $e,f,h\in\mathfrak{g}$ span a TDS. We assume that $adh $
has eigenvalues $-2,0,2$. Let $\mathfrak{g}_{j}=\{x\in\mathfrak{g}%
|[h,x]=jx\}$. Set $\mathfrak{m=g}_{0}$ and $\mathfrak{n}=\mathfrak{g}_{2}$ and
$\mathfrak{p}=\mathfrak{m}\oplus\mathfrak{n}$. Let $\chi(x)=(f,x)$ for
$x\in\mathfrak{n}$. Let $\mathcal{M}_{\chi}$ denote the full subcategory of
$\mathfrak{g}$-modules such that the elements $x-\chi(x)$ act locally
nilpotently for $x\in\mathfrak{n}$. If $M\in\mathcal{M}_{\chi}$ then we set
$M_{0}=\{0\}$ and $M_{j+1}=\{m\in M|(x-\chi(x))m\in M_{j}\}$ for $j\geq0$.
Then we have%
\[
\cup_{j\geq0}M_{j}=M,M_{j}\subset M_{j+1}\text{.}%
\]
We note that $M_{1}=\{m\in M|xm=\chi(x)m,x\in\mathfrak{n}\}.$

We define a bilinear form $q$ on $\mathfrak{n}$ by
\[
q(u,v)=-(adf^{2}u,v)=([f,u],[f,v]).
\]
since $adf^{2}:\mathfrak{g}_{2}\rightarrow\mathfrak{g}_{-2}$ is injective and
$(...,...)$ defines a perfect pairing between $\mathfrak{g}_{2}$ and
$\mathfrak{g}_{-2}$ we see that $q$ is nondegenerate and symmetric. Let
$x_{1},...,x_{d}$ be an orthonormal basis of $\mathfrak{n}$ relative to $q$.
We set%
\[
Q=\sum_{i=1}^{d}[f,x_{i}](x_{i}-\chi(x_{i})).
\]
Then $Q$ depends only on $e,f,h$ and the choice of $(...,...)$. We normalize
$(...,...)$ so that $(e,f)=1$.

\begin{lemma}
There exist elements $x_{ij}\in\mathfrak{n,}$ $1\leq i,j\leq d$ such that
$\chi(x_{ij})=0$ and $[x_{i},[f,x_{j}]]=\delta_{ij}e+x_{ij}$.
\end{lemma}

\begin{proof}
This follows directly form the definitions. Indeed, $([x_{i},[f,x_{j}%
]],f)=([f,x_{j}],[f,x_{i}])=\delta_{ij}$.
\end{proof}

If $v\in\mathfrak{n}$ we will use the notation $v^{\prime}$ for $v-\chi(v) $.

\begin{lemma}
If $u\in\mathfrak{m}$ then $uM_{j}\subset M_{j+1}$.
\end{lemma}

\begin{proof}
Let $z_{1},...,z_{j+1}\in\mathfrak{n}$. We must show that $z_{1}^{^{\prime}%
}\cdots z_{j+1}^{^{\prime}}um=0$ if $m\in M_{j}$. We write%
\[
z_{1}^{^{\prime}}\cdots z_{j+1}^{^{\prime}}um=uz_{1}^{^{\prime}}\cdots
z_{j+1}^{^{\prime}}m-\sum_{i=1}^{j+1}z_{1}^{\prime}\cdots\lbrack
u,z_{i}]\cdots z_{j+1}^{^{\prime}}m.
\]

Writing $[u,z_{i}]$ as $[u,z_{i}]^{\prime}+\chi([u,z_{i}])$ and using the fact
that $m\in M_{j}$ we have%
\[
z_{1}^{^{\prime}}\cdots z_{j+1}^{^{\prime}}um=-\sum_{i}\chi([u,z_{i}%
])z_{1}^{\prime}\cdots z_{i-1}^{\prime}z_{i+1}^{\prime}\cdots z_{j+1}%
^{^{\prime}}m=0.
\]

Which is $0$ for the same reason
\end{proof}

We now come to the main Lemma of this section (which due to our assumptions in
this section is a substantial simplification of the results in [W1]).

\begin{proposition}
If $m\in M_{s+1}$ then
\[
(Q-s)m\in M_{s}.
\]

\end{proposition}

\begin{proof}
We first note that $QM_{j}\subset M_{j}$ for all $j\geq0.$ This follows since
$[f,x_{i}]\in\mathfrak{m}$ and $(x_{i}-\chi(x_{i}))M_{j}\subset M_{j-1}$. Let
$1\leq i_{1},...,i_{s}\leq d$. If $m\in M_{s+1}$ we consider
\[
x_{i_{1}}^{^{\prime}}x_{i_{2}}^{^{\prime}}\cdots x_{i_{s}}^{^{\prime}}%
Qm=\sum_{j=1}^{d}x_{i_{1}}^{^{\prime}}x_{i_{2}}^{^{\prime}}\cdots x_{i_{s}%
}^{^{\prime}}[f,x_{j}]x_{j}^{\prime}m=
\]%
\[
\sum_{j=1}^{d}\sum_{l=1}^{s}x_{i_{1}}^{^{\prime}}x_{i_{2}}^{^{\prime}}\cdots
x_{i_{l-1}}^{^{\prime}}[x_{i_{l}},[f,x_{j}]]x_{i_{l+1}}^{\prime}\cdots
x_{i_{s}}^{\prime}x_{j}^{\prime}m+Qx_{i_{1}}^{^{\prime}}x_{i_{2}}^{^{\prime}%
}\cdots x_{i_{s}}^{^{\prime}}m.
\]
The last term above is $0$. We look at the sum. The lemma above implies that
$[x_{i_{l}},[f,x_{j}]]=\delta_{i_{l},j}e+x_{i_{l}j}$ with $\chi(x_{i_{l}%
j})=0.$ This implies that the sum is equal to%
\[
\sum_{j=1}^{d}\sum_{l=1}^{s}x_{i_{1}}^{^{\prime}}x_{i_{2}}^{^{\prime}}\cdots
x_{i_{l-1}}^{^{\prime}}\delta_{i_{l}j}ex_{i_{l+1}}^{\prime}\cdots x_{i_{s}%
}^{\prime}x_{j}^{\prime}m+\sum_{j=1}^{d}\sum_{l=1}^{s}x_{i_{1}}^{^{\prime}%
}x_{i_{2}}^{^{\prime}}\cdots x_{i_{l-1}}^{^{\prime}}x_{i_{l}j}x_{i_{l+1}%
}^{\prime}\cdots x_{i_{s}}^{\prime}x_{j}^{\prime}m.
\]
The latter sum is $0$ by the definition of $M_{s+1}$. Also since $e^{\prime
}=e-1$ we have the expression%
\[
\sum_{l=1}^{s}x_{i_{1}}^{^{\prime}}x_{i_{2}}^{^{\prime}}\cdots x_{i_{l-1}%
}^{^{\prime}}x_{i_{l+1}}^{\prime}\cdots x_{i_{s}}^{\prime}x_{i_{l}}^{\prime
}m=sx_{i_{1}}^{^{\prime}}x_{i_{2}}^{^{\prime}}\cdots x_{i_{s}}^{^{\prime}}m
\]
since $\mathfrak{n}$ is abelian. Thus $x_{i_{1}}^{^{\prime}}x_{i_{2}%
}^{^{\prime}}\cdots x_{i_{s}}^{^{\prime}}(Q-s)m=0$.
\end{proof}

If $M$ is a $\mathfrak{g}$-module we set $M_{1}=\{m\in M|(x-\chi(x))m=0\}.$ We
think of this as a functor from the category of $\mathfrak{g}$-modules to the
category of vector spaces. The next result is a special case of [W1] and is
original due to [L].

\begin{proposition}
The functor $M\rightarrow M_{1}$ is exact when restricted to $\mathcal{M}%
_{\chi}$.
\end{proposition}

\begin{proof}
Let
\[%
\begin{array}
[c]{ccccccccc}
&  &  &  &  & p &  &  & \\
0 & \rightarrow & A & \rightarrow & B & \rightarrow & C & \rightarrow & 0
\end{array}
\]
be an exact sequence in $\mathcal{M}_{\chi}$. Then we must show that if $v\in
C_{1}$ then there exists $w\in B_{1}$ such that $pw=v$. Let $z$ be any element
of $B$ such that $pz=v.$ Then if $z\in B_{1}$ we are done. Otherwise $z$ is in
$B_{s+1}$ for some $s\geq1$. The previous result implies that%
\[
\prod_{j=1}^{s}(Q-j)z\in B_{1}\text{.}%
\]
Also%
\[
p\prod_{j=1}^{s}(Q-j)z=\prod_{j=1}^{s}(Q-j)pz=(-1)^{s}s!v
\]
since $pz=v$ and $Qv=0$. Thus we may take $w=\frac{(-1)^{s}}{s!}\prod
_{j=1}^{s}(Q-j)z$.
\end{proof}

Let $M\in\mathcal{M}_{\chi}$ and let $F$ be a finite dimensional $\mathfrak{g}%
$-module. Let $F^{j}=\{f\in F|hf=jf\}$. If $x\in\mathfrak{n}$ then
$xF^{j}\subset F^{j+2}$. Let $r$ be the maximal value of $j$ such that
$F^{j}\neq0$. Thus the minimal value is $-r$. Then we have

\begin{proposition}
We use the notation above. $M_{1}\otimes F$ is contained in $(M\otimes
F)_{r+1}$. Furthermore,%
\[
\prod_{j=1}^{r}(Q-j):M_{1}\otimes F\rightarrow(M\otimes F)_{1}%
\]
bijectively.
\end{proposition}

\begin{proof}
If $z\in\mathfrak{n}$ and $m\in M_{1},f\in F^{j}$ then%
\[
z^{\prime}(m\otimes f)=z^{\prime}m\otimes f+m\otimes zf=m\otimes zf.
\]
This implies that
\[
z_{1}^{\prime}\cdots z_{r+1}^{\prime}(m\otimes f)=m\otimes z_{1}^{\prime
}\cdots z_{r+1}^{\prime}f
\]
and since $z_{1}^{\prime}\cdots z_{r+1}^{\prime}f\in F^{j+2r+2}$ and $j\geq-r$
the first assertion follows.

To prove the second assertion we first note that if $m\in M_{1}$ and $f\in
F^{j}$ then%
\begin{equation}
Q(m\otimes f)=\sum[f,x_{i}](m\otimes x_{i}f).\label{eq:qu}%
\end{equation}
Thus if we set $F_{\geq j}=\oplus_{l\geq j}F^{l}$. Then
\[
Q(m\otimes F_{\geq j})\subset M\otimes F_{\geq j+2}.
\]
Set $\Gamma=$ $\prod_{j=1}^{r}(Q-j)$. Let $u=\sum m_{ji}\otimes f_{ji}$ with
$m_{ji}\in M_{1}$ and $f_{ji}\in F^{j}$ forming a basis of $F^{j}$. Let
$j_{o}$ be the minimum of the $j$ such that there exists an $m_{ji}\neq0 $.
Suppose that $\Gamma u=0$. Then (\ref{eq:qu}) implies that
\[
0=(-1)^{r}r!\sum_{i}m_{j_{o}i}\otimes f_{j_{o}i}+\sum_{i,j>j_{o}}w_{ji}\otimes
f_{ji}%
\]
for appropriate $w_{ji}$. This is a contradiction. So $\Gamma$ is injective. We
note that by the first part and Lemma 3, $\Gamma(M_{1}\otimes F)\subset
(M\otimes F)_{1}$. To complete the proof we need only prove the surjectivity.
So let $u\in(M\otimes F)_{1}$. Write $u$ as $\sum_{i,j\geq j_{o}}m_{ji}\otimes
f_{ji}$ as above. Then (*) above implies that $m_{j_{o}i}\in M_{1}$. Thus if
$v_{j_{o}}=\sum_{i}m_{j_{o}i}\otimes f_{j_{o}i} $ then (as above)
\[
\frac{(-1)^{r}}{r!}\Gamma v_{j_{o}}=v_{j_{o}}+\sum_{i,j>j_{o}}w_{ji}\otimes
f_{ji}.
\]
Hence $z=u-\frac{(-1)^{r}}{r!}\Gamma v_{j_{o}}\in(M\otimes F)_{1}$ and
\[
z=\sum_{i,j\geq j_{1}}n_{ji}\otimes f_{ji}%
\]
with $j_{1}>j_{o}$. We can thus prove the surjectivity by the obvious
contradiction or by the obvious iteration to calculate a preimage.
\end{proof}

\section{Some Bruhat theory.} \label{sec:bruhat}

Let $G$ be a real reductive group of inner type (as in [W3]) with compact
center and maximal compact subgroup $K.$ Let $P_{o}$ be a minimal parabolic
subgroup of $G$ and let $P\supset P_{o}$ be a parabolic subgroup. We endow the
two groups with compatible Langlands decompositions $P_{o}=M_{o}A_{o}N_{o}$
and $P=MAN$. That is, $M_{o}A_{o}$ (resp. $MA)$ is a Levi factor of $P_{o}$
(resp. $P)$, $A_{o}$ is the identity component of a maximal split torus over $%
\mathbb{R}
$, $A$ is the identity component of a maximal split torus in the center of
$P\cap\theta P$ ($\theta$ the Cartan involution corresponding to $K$) and
$N_{o}$ (resp. $N$) is the nilradical of $P_{o}$ (resp. $P$). We have
$N\subset N_{o},A\subset A_{o},M_{o}\subset M.$ Let $W$ be the Weyl group of
$G$ with respect to $A_{o}$ and let $W_{M}$ be the Weyl group of $MA$ with
respect to $A_{o}$. Let $\Phi$ (resp. $\Phi_{M}$) denote the root system of
$G$ (resp. $MA)$ relative to $A_{o}.$ Let $\Phi^{+}$ denote the system of
positive roots in $\Phi$ that are roots of $A_{o}$ on $N_{o}$. Set $\Phi
_{M}^{+}=\Phi^{+}\cap\Phi_{M}$.

Let $W^{M}$ denote the set of all $w\in W$ such that $w\Phi_{M}^{+} \subset \Phi^{+}$. Then it is clear that $W=W^{M}W_{M}$ with unique decomposition. If
$w\in W$ we choose $w^{\ast}\in N_{K}(A_{o})$ such that $M_{o}w^{\ast}=w$
(here $N_{K}(A_{o})$ is the normalizer of $A_{o}$ in $K$). The Bruhat Lemma
says that
\[
G=\bigcup_{w\in W^{M}}P_{o}w^{\ast}P\text{.}%
\]

\begin{lemma}\label{lemma:bruhat}
With the notation as above and $K_{M}=K\cap M$ we have%
\[
G=\bigcup_{w\in W^{M}}P_{o}w^{\ast}K_{M}N.
\]
Furthermore, let $w_{G}$ denote the longest element of $W$ and $w_{M}$ the
longest element of $W_{M}$ then if $w\in W^{M}$ and $w\neq w^{M}=w_{G}w_{M}$
then
\[
\dim P_{o}w^{\ast}K_{M}N<\dim P_{o}(w_{G}w_{M})^{\ast}K_{M}N.
\]

\end{lemma}

\begin{proof}
We note that $W^{M}=\{w\in W|w\Phi_{M}^{+}\subset\Phi^{+}\}$ hence
$w^{\ast}(P_{o}\cap M)(w^{\ast})^{-1}\subset P_{o}$. The Iwasawa decomposition
implies that%
\[
M=(P_{o}\cap M)K_{M}.
\]
Since $w^{\ast}A_{o}(w^{\ast})^{-1}\subset A_{o}$ for all $w\in W$ we see that%
\[
G=\bigcup_{w\in W^{M}} P_{o} w^{\ast} (P_{o}\cap M)K_{M}N =\bigcup_{w\in W^{M}%
}P_{o}w^{\ast}(P_{o}\cap M)(w^{\ast})^{-1}w^{\ast}K_{M}N%
\]%
\[
=\bigcup_{w\in W^{M}}P_{o}w^{\ast}K_{M}N.
\]
The dimension assertion follows from the fact that $w=w_{G}w_{M}$ is the
unique element of $W^{M}$ such that $w^{\ast}N(w^{\ast})^{-1}\cap N_{o}=\{1\}$.
\end{proof}

We now return to the notation in section \ref{sec:examples}. We choose a minimal parabolic
subalgebra in $P$, $P_{o}$, and $K,e,f,h$ as in that section. Then
$Lie(K_{M})=\mathfrak{m}^{f}$. Let $(...,...)$ denote a positive multiple of
the Killing form. Let $\chi(\exp(x))=e^{itB(f,x)}$ with $t\neq0$ for $x\in\mathfrak{n}$. We fix
an admissible, finitely generated, smooth Fr\'{e}chet representation of $M$ of
moderate growth, $(\sigma,V_{\sigma})$. Let $\nu\in\mathfrak{a}_{%
\mathbb{C}
}^{\ast}$ ($\mathfrak{a}=Lie(A)$). We set
\[
I_{P,\sigma,\nu}^{\infty}=\left\lbrace \phi:G\rightarrow V_{\sigma}\, \left| \, \begin{array}{c} \mbox{$\phi$ is $C^{\infty}$, $\phi(namg)=a^{\nu+\rho
}\sigma(m)\phi(g)$}, \\ \mbox{$n\in N$, $a\in A$, $m\in M$, $g\in G$}\end{array}\right.\right\rbrace
\]
and $\pi_{P,\sigma,\nu}(g)\phi(x)=\phi(xg)$. If we endow $I_{P,\sigma,\nu
}^{\infty}$ with the usual $C^{\infty}$ topology then $\pi_{P,\sigma,\nu}$ is
an admissible, finitely generated smooth Fr\'{e}chet representation of
moderate growth. We set $U_{P,\sigma,\nu}=\{\phi\in I_{P,\sigma,\nu}^{\infty
}|$supp$\phi\subset P(w^{M})^{\ast}N\}$. Let $Wh_{\chi}(I_{P,\sigma,\nu
}^{\infty})$ be the space of all continuous functionals $T\in
(I_{P,\sigma,\nu}^{\infty})^{\prime}$ (continuous dual) such that
$T(\pi_{p,\sigma,\nu}(n)^{-1}\phi)=\chi(n)T(\phi)$, this space is called the space of Bessel models for $I_{P,\sigma,\nu}^{\infty}$. We will now show how to define a $K_{M}$ intertwining operator
\[
 \Phi_{P,\sigma,\nu}:Wh_{\chi}(I^{\infty}_{P,\sigma,\nu}) \longrightarrow V_{\sigma}'.
\]
Let ${\D}'(P(w^M)^{\ast}N:V_{\sigma})$ be the space of $V_{\sigma}'$-valued distributions on $P(w^M)^{\ast}N$. Given $\lambda \in Wh_{\chi}(I^{\infty}_{P,\sigma,\nu})$, we will first define a distribution $\bar{\lambda}\in{\D}'(P(w^M)^{\ast}N:V_{\sigma})$ in the following way: given $f\in C_{c}^{\infty}(P(w^M)^{\ast}N)$ and $v \in V_{\sigma}$, extend $f$ to $G$ by $0$ in the complement of $P(w^M)^{\ast}N$, and define $\bar{f}_{v} \in I^{\infty}_{P,\sigma,\nu}$ by
\[
 \bar{f}_{v}(g)=\int_{P}\sigma_{\nu}(p)^{-1}\delta_{P}(p)^{-\frac{1}{2}}f(pg)v\, d_{r}p
\]
for a choice of a right invariant measure $d_{r}p$. Observe that $\bar{f}_{v}\in U_{P,\sigma,\nu}$. Now define
\[
 \bar{\lambda}(f)(v)=\lambda(\bar{f}_{v}).
\]
It's easy to check that $\bar{\lambda}(L_{p}f)=[\delta_{P}(p)^{\frac{1}{2}}\sigma_{\nu}(p)^{\ast}](\bar{\lambda}(f))$, and $\bar{\lambda}(R_{n}f)=\chi(n)^{-1}\bar{\lambda}(f)$, i.e., $\bar{\lambda}$ is actually in ${\D}'(P(w^M)^{\ast}N:V_{\sigma_{\nu-\rho}}\otimes \mathbb{C}_{\chi^{-1}})^{P\times N}$, the space of $(V_{\sigma_{\nu-\rho}}\otimes \mathbb{C}_{\chi^{-1}})'$-valued invariant distributions under the action of $P\times N$. Observe that $P \times N$ acts transitively on $P(w^M)^{\ast}N$ and hence by theorem 3.11 of [K-V] there exists an isomorphism
\[
 {\D}'(P(w^M)^{\ast}N:V_{\sigma_{\nu-\rho}}\otimes \mathbb{C}_{\chi^{-1}})^{P\times N} \longrightarrow (V_{\sigma_{\nu-\rho}}'\otimes \mathbb{C}_{\chi^{-1}})^{(P \times N)_{(w^M)^{\ast}}}\cong V_{\sigma}'.
\]
Let $\mu_{\lambda}$ be the image of $\bar{\lambda}$ under this isomorphism. Then using again theorem 3.11 of [K-V] we have the following explicit description of $\bar{\lambda}$ in terms of $\mu_{\lambda}$:
\begin{eqnarray}
\bar{\lambda}(f)(v) & = & \int_{N}\int_{P}\chi(n)^{-1}f(p(w^M)^{\ast}n)\mu_{\lambda}(\sigma_{\nu}(p)^{-1}\delta_{P}(p)^{-\frac{1}{2}} v)\, d_{r}p\, dn \nonumber \\
                     & = & \mu_{\lambda}\circ J_{P,\sigma,\nu}^{\chi}(\bar{f}_{v})=\lambda(\bar{f}_{v}). \label{eq:KV}
\end{eqnarray}
Where the last equality follows from the definition of $\bar{\lambda}$. Set $\Phi_{P,\sigma,\nu}(\lambda)=\mu_{\lambda}$. Observe that, since $\bar{f}_{v}\in U_{P_{\circ},\eta}$, the integral in the right hand side of equation (\ref{eq:KV}) is absolutely convergent, and hence $\Phi_{P,\sigma,\nu}$ is well defined for all $\nu \in \mathfrak{a}'$.  The purpose of this section is to show that $\Phi_{P,\sigma,\nu}$ is injective, in the next section we will show that it is surjective. The following lemma is an important first step in order to show the injectivity of $\Phi_{P,\sigma,\nu}$.

\begin{lemma}
If $w\in W^{M}$ is not $w^{M}$ then
\[
(f,Ad(w^{\ast})Lie(N_{o}))\neq\{0\}.
\]

\end{lemma}

\begin{proof}
We note that the tube type assumption in section \ref{sec:examples} also implies that $\Phi$ is
a root system of type $C_{n}$ with $n=\dim A_{o}$. This implies that there
exist linear functionals $\varepsilon_{1},...,\varepsilon_{n}$ on
$\mathfrak{a}_{o}=Lie(A_{o})$ such that
\[
\Phi^{+}=\{\varepsilon_{i}\pm\varepsilon_{j}|1\leq i<j\leq n\}\cup
\{2\varepsilon_{1},...,2\varepsilon_{n}\}
\]
and%
\[
\Phi_{M}^{+}=\{\varepsilon_{i}-\varepsilon_{j}|1\leq i<j\leq n\}.
\]
We note that this implies that $\varepsilon_{i}(h)=1$ for all $i$. If $w\in
W^{M}$ and $(f,Ad(w^{\ast})Lie(N_{o}))=\{0\}$ then we must have%
\[
-2\varepsilon_{i}\in w\Phi^{+},\qquad i=1,...,n.
\]
We therefore see that $\varepsilon_{i}(w^{-1}h)<0$ for all $i$. Hence
$-(\varepsilon_{i}+\varepsilon_{j})\in w\Phi^{+}$ for all $i\leq j$. This
implies that $w=w^{M}$.
\end{proof}

We will now use these results and some standard Bruhat theory to prove some
vanishing results. Observe that the next proposition immediately implies the injectivity of $\Phi_{P,\sigma,\nu}$.

\begin{proposition}\label{prop:upperbound}
If $\lambda \in Wh_{\chi}(I_{P,\sigma,\nu}^{\infty})$ and $\lambda_{|U_{P,\sigma,\nu}}=0$
then $\lambda=0$.
\end{proposition}

\begin{proof}
We will first reduce the problem to the case where $\sigma$ is an induced representation. If $(\eta,V)$ is a finite dimensional representation of $P_{\circ}$, we define $I^{\infty}_{P_{\circ},\eta}$ to be the space of smooth $\phi:G \longrightarrow V$ such that $\phi(pg)=\eta(p)\phi(g)$ for $p\in P_{\circ}$ and $g\in G$. Set $\pi_{\eta}(g)\phi(x)=\phi(xg)$, for $x,g\in G$. If we endow $I^{\infty}_{P_{\circ},\eta}$ with the $C^{\infty}$ topology, then $(\pi_{\eta},I^{\infty}_{P_{\circ},\eta})$ is an admissible, finitely generated smooth Fr\'echet representation of moderate growth. Let $(\xi,F)$ be a finite dimensional representation of $P_{M}:=P_{\circ}\cap M$ and let $(\pi_{\xi},I^{\infty}_{P_{M},\xi})$ be the corresponding representation of $M$ (Observe that $P_{\circ}\cap M$ is a minimal parabolic subgroup of $M$). The Casselman-Wallach theorem implies that there exists a surjective, continuous, $M$-intertwining operator $L:I^{\infty}_{P_{M},\xi} \longrightarrow V_{\sigma}$ for some finite dimensional representation $(\xi, F)$ of $P_{M}$. This map lifts to a surjective $G$ interwining map $\tilde{L}:I^{\infty}_{P,\pi_{\xi},\nu} \longrightarrow I^{\infty}_{P,\sigma,\nu}$, given by $\tilde{L}(\phi)(g)=L(\phi(g))$ for $\phi \in I^{\infty}_{P,\pi_{\xi},\nu}$, $g\in G$. The representation  $I^{\infty}_{P,\pi_{\xi},\nu}$ is equivalent to the representation smoothly induced from $P_{\circ}$ to $G$ by the representation $\xi_{\nu}$ of $P_{\circ}$ with values on $F$ defined as follows:
\[
 \xi_{\nu}(nap)=a^{\nu+\rho}\xi(p) \qquad \mbox{for $p\in P_{M}$, $a\in A$, $n\in N$.}
\]
Setting $\eta=\xi_{\nu}$ we can identify the map $\tilde{L}$ with a surjective $G$-equivariant map $\tilde{L}:I^{\infty}_{P_{\circ},\eta} \longrightarrow I^{\infty}_{P,\sigma,\nu}$. Set $U_{P_{\circ},\eta}=\{\phi\in I^{\infty}_{P_{\circ},\eta} \, | \, \mbox{$\operatorname{supp} \phi \subset P(w^{M})^{\ast}N$}\}$, and define $Wh_{\chi}(I^{\infty}_{P_{\circ},\eta})$ in the same way as above. Assume that we have proved the proposition for $I^{\infty}_{P_{\circ},\eta}$, i.e., assume that if $\lambda \in Wh_{\chi}(I^{\infty}_{P_{\circ},\eta})$ and $\lambda|_{U_{P_{\circ},\eta}}=0$, then $\lambda=0$. Let $\lambda\in Wh_{\chi}(I^{\infty}_{P,\sigma,\nu})$ be such that $\lambda|_{U_{P,\sigma,\nu}}=0$ and let $\tilde{\lambda}=\tilde{L}^{\ast}\lambda$ be the pullback of $\lambda$ to $Wh_{\chi}(I^{\infty}_{P_{\circ},\eta})$ by $\tilde{L}$. It's easy to check, using the definition of $\tilde{L}$, that $\tilde{\lambda}|_{U_{P_{\circ},\eta}}=0$ and hence, by our assumptions, $\tilde{\lambda}=\tilde{L}^{\ast}\lambda=0$, but $\tilde{L}$ is surjective, therefore $\lambda=0$. We will now prove the proposition for $Wh_{\chi}(I^{\infty}_{P_{\circ},\eta})$.

Let $D'(G:F)$ be the space of all $F'$-valued distributions on $G$, i.e., the space of all continuous maps form $C_{c}^{\infty}(G)$ to $F'$. Let $\lambda \in Wh_{\chi}(I^{\infty}_{P_{\circ},\eta})$ be such that $\lambda|_{U_{P_{\circ},\eta}}=0$. We will define a distribution $\bar{\lambda}\in D'(G:F)$ in the following way: given $\phi\in C_{c}^{\infty}(G)$, and $v\in V_{\sigma}$, define $\bar{\phi}_{\nu}\in I^{\infty}_{P_{\circ},\eta}$ by
\[
 \bar{\phi}_{\nu}(g)=\int_{P}\eta(p)^{-1}\phi(pg)v\, d_{r}p,
\]
 and set $\bar{\lambda}(\phi)(v)=\lambda(\bar{\phi}_{\nu})$. It's easy to check that $\bar{\lambda}(L_{n_{\circ}}\phi)=[\eta(n_{\circ})^{\ast}]\bar{\lambda}(\phi)$ and $\bar{\lambda}(R_{n}\phi)=\chi(n)^{-1}\bar{\lambda}(\phi)$, i.e., $\bar{\lambda}$ is actually in $D'(G:F\otimes \mathbb{C}_{\chi})^{N_{\circ}\times N}$ ($\chi(n)^{\ast}=\chi(n)^{-1}$). We note that $K_{M}=M^{f}$, thus $K_{M}=\{m\in M \, | \, \chi \circ Ad(m)=\chi\}$. Hence we can extend the action of $N_{\circ}\times N$ on $F'\otimes \mathbb{C}_{\chi^{-1}}$ to an action of $P_{\circ}\times K_{M}N$. Furthermore, according to lemma \ref{lemma:bruhat}, $G$ has a finite orbit decomposition under the action of $P_{\circ}\times K_{M}N$. Hence if we set $H=P_{\circ}\times K_{M}N$, and $H'=N_{\circ}\times N$ we are in the setting of the vanishing theorem 3.9 of [KV]. To finish the proof of the proposition we just need to show that $\bar{\lambda}=0$. Observe that, since $\lambda|_{U_{P,\sigma,\nu}}=0$, $\bar{\lambda}=0$ on the big Bruhat cell $P(w^{M})^{\ast}N$. The proof of the vanishing of $\bar{\lambda}$ is now completely analogous to the proof of theorem 4.1 in [KV] and follows from the fact that $N_{\circ}$ acts unipotently on $U(\mathfrak{g})\otimes F'$ and $\chi$ is non-trivial on $N\cap w^{\ast}N_{\circ}(w^{\ast})^{-1}$ if $w\in W^{M}$ and $w\neq w^{M}$.
\end{proof}

\section{The main theorem}

We are now ready to combine all of the material of the previous sections. The
$e,f,h$ in section \ref{sec:tensoring} do not agree with those in sections \ref{sec:examples} and \ref{sec:bruhat}. In the
formulas in section 1 we must use $e\rightarrow\frac{1}{it}e,f\rightarrow itf$
and an appropriate normalization of $B$. With this in mind we denote by $Q$
the operator corresponding to these choices. We retain all of the notation of
the previous sections (with the proviso above). Let $F$ be a finite
dimensional representation of $G$. Let $r$ be as in Proposition 6. Set
\[
\Gamma_{F}=\frac{(-1)^{r}}{r!}\prod_{i=1}^{r}(Q-i).
\]
Then Proposition 6 says that if $M\in\mathcal{M}_{\chi}$ then
\[
\Gamma_{F}:M_{1}\otimes F\rightarrow(M\otimes F)_{1}%
\]
is bijective. We also note that the definition of $Q$ immediately implies

\begin{lemma}
With the notation above $Q\in U(\mathfrak{g}_{%
\mathbb{C}
})^{K_{M}}$.
\end{lemma}

We are now left with the proof of the surjectivity of $\Phi_{P,\sigma,\nu}$. As in [W1] we will first show that $\Phi_{P,\sigma,\nu}$ is a bijection for $\nu$ in an open set of $\mathfrak{a}'$, and then we will use the results in section \ref{sec:tensoring} to show that $\Phi_{P,\sigma,\nu}$ is an isomorphism for all $\nu \in \mathfrak{a}'$

\begin{lemma}\label{lemma:hyperplane}
 There exists a constant $q_{\sigma}$ such that
\[
 J_{P,\sigma,\nu}^{\chi}(\phi)=\int_{N}\chi(n)^{-1}\phi_{P,\sigma,\nu}((w^M)^{\ast}n)\, dn
\]
converges absolutely and uniformly in compacta of $\{\nu \in \mathfrak{a}^{\ast}_{\mathbb{C}}|\operatorname{Re} \nu > q_{\sigma}\}$ for all $\phi\in I^{\infty}_{K\cap M}$. Furthermore, if $\nu\in \mathfrak{a}^{\ast}_{\mathbb{C}}$, $\operatorname{Re} \nu > q_{\sigma}$, and $\mu \in V_{\sigma}'$, $\mu \neq 0$, then $\mu \circ J^{\chi}_{P,\sigma,\nu}\neq 0$ and
\[
\Phi_{P,\sigma,\nu}:Wh_{\chi}(I^{\infty}_{P,\sigma,\nu}) \longrightarrow V_{\sigma}'
\]
is an isomorphism.
\end{lemma}
\begin{proof}
 The convergence result can be proved using the same arguments as in [W1]. To prove the isomorphism statement observe that, since $J_{P,\sigma,\nu}^{\chi}(\phi)$ is absolutely convergent for $\mbox{Re}\,\nu > q_{\sigma}$, we can use equation (\ref{eq:KV}) to define the inverse of $\Phi_{P,\sigma,\nu}$.
\end{proof}

\begin{theorem}
With notation and assumptions as in the end of last section.

\begin{description}
 \item[i)] The map
\[
\Phi_{P,\sigma,\nu}:Wh_{\chi}(I_{P,\sigma,\nu}^{\infty}) \longrightarrow V_{\sigma}'
\]
defines a $K_{M}$-equivariant isomorphism for all $\nu \in \mathfrak{a}^{\ast}_{\mathbb{C}}$.
\item[ii)] For all $\mu\in V_{\sigma}'$ the map $\nu\mapsto \mu\circ J_{P,\sigma,\nu}^{\chi}$ extends to a weakly holomorphic map of $\mathfrak{a}^{\ast}_{\mathbb{C}}$ into $(I^{\infty}_{K\cap M,\sigma|_{K\cap M}})'$.
\end{description}
\end{theorem}
This result immediately implies
\begin{corollary} Let $F_{\tau}$ be an irreducible, finite dimensional representation of $K_{M}$, and let
$Wh_{\chi,\tau}(I^{\infty}_{P,\sigma,\nu})$ be the space of all continuous linear maps $T:I_{P,\sigma,\nu}^{\infty}\rightarrow F_{\tau}$ such that
\[
 T(\pi_{P,\sigma,\nu}(kn)\phi)=\chi(n)^{-1}\tau(k)T(\phi)
\]
for $k\in K_{M}$, $n\in N$. Then
 \[
  \dim Wh_{\chi,\tau}(I^{\infty}_{P,\sigma,\nu})=\dim Hom_{K_{M}}(V_{\sigma},F_{\tau})
 \]
\end{corollary}

\begin{proof}[Proof (of theorem)] During the course of the proof we will use the following notation: If $(\eta,V_{\eta})$ is a representation of $P$, let
\begin{eqnarray*}
 I_{P,\xi}^{\infty} & = & \{\phi:G\longrightarrow V_{\xi}\, | \, \phi(pg)=\delta(p)\xi(p)\phi(g)\} \\
          U_{P,\xi} & = & \{\phi \in I_{P,\xi}^{\infty}\, | \, \operatorname{supp} \phi \subset P(w^{M})^{\ast}N\}.
\end{eqnarray*}
Observe that if $\phi\in U_{P,\xi}$ then $\phi$ has compact support modulo $P$, as otherwise the support of $\phi$, which is a closed set, could not be completely contained in the open cell $P(w^{M})^{\ast}N$. Define
\[
 \Phi_{P,\xi}: Wh_{\chi}(I_{P,\xi}^{\infty}) \longrightarrow V_{\xi}'
\]
to be the map analogous to the map $\Phi_{P,\sigma,\nu}$ that was defined in the last section.

Let $L$ be such that if $\operatorname{Re} \nu > L$, then $\Phi_{P,\sigma_{\nu}}$ is an isomorphism. Observe that, by lemma \ref{lemma:hyperplane}, such an $L$ exists. We will fix such a $\nu$.

Let $(\eta,F)$ be a finite dimensional representation of $G$, and let
\[
 F_{j}=\{v \in F | h \cdot v = jv\}.
\]
Then $F=\oplus_{j=0}^{r}F_{r-2j}$, for some $r$ with $F_{r} \neq (0)$. Note that $\mathfrak{n}F_{r-2j}\subset F_{r-2(j+1)}$ and hence if we set $X_{j}= \oplus_{k=j}^{r}F_{r-2k}$, then
\begin{equation}
 F=X_{0}\supset X_{1}\supset \cdots \supset X_{r}\supset X_{r+1}=(0) \label{eq:filtration}
\end{equation}
is a $P$-invariant filtration. Let $Y^{j}=\{\phi \in F'|\phi|_{X_{j}=0}\}$, then
\[
 F'=Y^{r+1}\supset Y^{r}\supset \cdots \supset Y^{0}=(0)
\]
is the dual filtration to (\ref{eq:filtration}).

We will take $I_{P,\sigma_{\nu}}^{\infty}\otimes F$ to be the space
\[
 \{\phi:G \longrightarrow V_{\sigma_{\nu}}\otimes F |\phi(pg)=(\sigma_{\nu}(p)\otimes I)\phi(g) \}
\]
and observe that
\[
 I_{P,\sigma_{\nu}\otimes\eta}^{\infty}=\{\phi:G \longrightarrow V_{\sigma_{\nu}}\otimes F |\phi(pg)=(\sigma_{\nu}(p)\otimes \eta(p))\phi(g)\}.
\]
With this conventions there is an isomorphism of $G$-modules
\begin{eqnarray}
 I_{P,\sigma_{\nu}}^{\infty}\otimes F &\cong & I_{P,\sigma_{\nu}\otimes\eta}^{\infty} \label{eq:identification} \\
                                   \phi & \mapsto & \hat{\phi} \nonumber \\
                                   \check{\phi} & \leftarrow & \phi \nonumber
\end{eqnarray}
where
\[
 \hat{\phi}(g)=(I\otimes\eta(g))\phi(g), \qquad \mbox{and} \qquad \check{\phi}(g)=(I\otimes\eta(g)^{-1})\phi(g).
\]
Let $(\eta_{j},X_{j})$ be the restriction of $\eta$ to $P$ acting on $X_{j}$, and let $(\bar{\eta}_{j},X_{j}/X_{j+1})$ be the representation induced on the quotient. The $P$-invariant filtration given in (\ref{eq:filtration}) induces via the isomorphism in (\ref{eq:identification}) a $G$-invariant filtration
\[
 I_{P,\sigma_{\nu}\otimes\eta}^{\infty}=I_{P,\sigma_{\nu}\otimes\eta_{0}}^{\infty}\supset\ldots\supset I_{P,\sigma_{\nu}\otimes\eta_{r+1}}^{\infty}=(0)
\]
and it's easy to check that
\[
 I_{P,\sigma_{\nu}\otimes\eta_{j}}^{\infty}/I_{P,\sigma_{\nu}\otimes\eta_{j+1}}^{\infty} \cong I_{P,\sigma_{\nu}\otimes\bar{\eta}_{j}}^{\infty}.
\]

Let $\Gamma=(-1)^{r}/r!\, \prod_{j=1}^{r}(Q-j)$ as before, and define
\[
 \check{\Gamma}:Wh_{\chi}(I_{P,\sigma_{\nu}}^{\infty}) \otimes F' \longrightarrow Wh_{\chi}(I_{P,\sigma_{\nu}\otimes \eta}^{\infty})
\]
by $\check{\Gamma}(\lambda)(\phi)=\Gamma(\lambda)(\check{\phi})$. Then $\check{\Gamma}$ defines a $K_{M}$-equivariant isomorphism. Now since $\Phi_{P,\sigma_{\nu}}$ is an isomorphism, we can define $\tilde{\Gamma}$ such that the following diagram is commutative
\begin{equation}\label{diag1}
 \begin{picture}(250,155)
 \put(20,10){$V_{\sigma}'\otimes F'$}
 \put(170,65){$(V_{\sigma}\otimes F)'$}
\put(0,120){$Wh_{\chi}(I_{P,\sigma_{\nu}}^{\infty})\otimes F'$}
\put(170,120){$Wh_{\chi}(I_{P,\sigma_{\nu}\otimes\eta}^{\infty})$}
\put(175,10){$V_{\sigma}'\otimes F'$}
\put(195,113){\vector(0,-1){35}}
\put(195,58){\vector(0,-1){35}}
\put(90,125){\vector(1,0){70}}
\put(75,15){\vector(1,0){85}}
\put(40,113){\vector(0,-1){90}}
\put(120,130){$\check{\Gamma}$}
\put(120,20){$\tilde{\Gamma}$}
\put(45,65){$\Phi_{P,\sigma_{\nu}}\otimes id$}
\put(200,95){$\Phi_{P,\sigma_{\nu}\otimes \eta}$}
\end{picture}
\end{equation}

\begin{claim}\label{claim}
 $\tilde{\Gamma}$ is an isomorphism.
\end{claim}

Observe that the claim immediately implies that $\Phi_{P,\sigma_{\nu}\otimes\eta}$ is an isomorphism. We will now proceed to prove the claim.

Let $\mu\in V_{\sigma}'\otimes F'$. Since we are assuming that $\Phi_{P,\sigma_{\nu}}$ is an isomorphism, there exists $\lambda \in Wh_{\chi}(I_{P,\sigma_{\nu}}^{\infty})\otimes F'$ such that $\mu=(\Phi_{P,\sigma_{\nu}}\otimes I)(\lambda)=:\mu_{\lambda}$, i.e., if $\check{\phi}\in U_{P,\sigma_{\nu}\otimes\eta}$, then
\[
 \lambda(\phi)=\mu_{\lambda}(\int_{N}\chi(n)^{-1}\phi(w_{M}n)\, dn ).
\]
We will show that if $\mu_{\lambda}\in V_{\sigma}'\otimes Y^{j}$, then $\mu_{\lambda}-(I\otimes\eta(w_{M})^{\ast})\tilde{\Gamma}(\mu_{\lambda})\in V_{\sigma}'\otimes Y^{j-1}$, which is enough to prove the claim. Observe that for any such $\mu_{\lambda}$, $\lambda \in Wh_{\chi}(I_{P,\sigma_{\nu}}^{\infty})\otimes Y^{j}$, and hence $\Gamma(\lambda)=\lambda + \tilde{\lambda}$, with $\tilde{\lambda}\in (I_{P,\sigma_{\nu}}^{\infty})'\otimes Y^{j-1}$. Therefore if $\phi\in I_{P,\sigma_{\nu}}^{\infty} \otimes X_{j-1}$, then $\tilde{\lambda}(\phi)=0$ and
\[
 \Gamma(\lambda)(\phi)=\lambda(\phi)+\tilde{\lambda}(\phi)=\lambda(\phi).
\]
Now by definition
\begin{equation} \label{eq:Gamma-lambda}
 \check{\Gamma}(\lambda)(\hat{\phi})=\Gamma(\lambda)(\phi)=\mu_{\lambda}(\int_{N}\chi(n)^{-1}\phi(w_{M}n)\, dn ).
\end{equation}
But on the other hand since $\hat{\phi}\in U_{P,\sigma_{\nu}\otimes\eta}$
\begin{eqnarray}\label{eq:Gamma-lambda-check}
 \check{\Gamma}(\lambda)(\hat{\phi})& = &\mu_{\check{\Gamma}(\lambda)}(\int_{N}\chi(n)^{-1}\hat{\phi}(w_{M}n)\, dn ) \\ \nonumber
                                 & = & \tilde{\Gamma}(\mu_{\lambda})(\int_{N}\chi(n)^{-1}(I\otimes\eta(w_{M}n))\phi(w_{M}n)\, dn ). \nonumber
\end{eqnarray}
Let $v_{j-1}\in V_{\sigma}\otimes X^{j-1}$. We can always find $\phi\in I^{\infty}_{P,\sigma_{\nu}}\otimes X_{j-1}$ such that $\phi(w_{M})=v_{j-1}$. Let $\{u^{l}\}_{l}$ be an approximate identity on $N$, i.e., $\{u^{l}\}_{l}\subset C_{c}^{\infty}(N)$, $\int_{N} u_{l}(n) \, dn =1$ for all $l$, and
\[
 \lim_{l\rightarrow \infty} \int_{N} \phi(n)u_{l}(n) \, dn = \phi(e) \qquad \mbox{for all $\phi \in C_{c}^{\infty}(N)$.}
\]. Define $\phi^{l}\in U_{P,\sigma_{\nu}}$ by
\[
 \phi^{l}(w_{M}n)=u^{l}(n)\chi(n)\phi(w_{M}n)
\]
Then from (\ref{eq:Gamma-lambda}) and (\ref{eq:Gamma-lambda-check})
\begin{eqnarray*}
 \lim_{l\rightarrow \infty} \mu_{\lambda}(\int_{N}\chi(n)^{-1}\phi^{l}(w_{M}n)\, dn ) & = & \lim_{l\rightarrow \infty} \tilde{\Gamma}(\mu_{\lambda})(\int_{N}\chi(n)^{-1}\hat{\phi}(w_{M}n)\, dn ) \\
 \mu_{\lambda}(\lim_{l\rightarrow \infty} \int_{N}u^{l}(n)\phi(w_{M}n)\, dn ) & = & \tilde{\Gamma}(\mu_{\lambda})(\lim_{l\rightarrow \infty} \int_{N}u^{l}(n)\hat{\phi}(w_{M}n)\, dn ) \\
\mu_{\lambda}(\phi(w_{M}))& = & \tilde{\Gamma}(\mu_{\lambda})(I\otimes\eta(w_{M}))\phi(w_{M})),
\end{eqnarray*}
where in the last step we are using that $\{u^{l}\}_{l}$ is an approximate identity on $N$. We can rewrite the above equation as
\begin{equation}
 \mu_{\lambda}(v_{j-1})  = [(I\otimes\eta(w_{M}^{-1})^{\ast})\tilde{\Gamma}(\mu_{\lambda})](v_{j-1}). \label{eq:mu-lambda-v}
\end{equation}
Since equation (\ref{eq:mu-lambda-v}) holds for all $v_{j-1}\in V_{\sigma}\otimes X_{j-1}$ it follows that $\mu_{\lambda}-(I\otimes\eta(w_{M}^{-1})^{\ast})\tilde{\Gamma}(\mu_{\lambda}) \in V_{\sigma}' \otimes Y^{j-1}$ as desired, which concludes the proof of the claim \ref{claim}.

Now let
\[
 W^{j}=\{\lambda \in Wh_{\chi}(I_{P,\sigma_{\nu}\otimes\eta}^{\infty}) | \lambda|_{I_{P,\sigma_{\nu}\otimes \eta_{j}}^{\infty}}=0\}.
\]
Observe that if $\lambda \in W^{j+1}$, then $ \lambda|_{I_{P,\sigma_{\nu}\otimes \eta_{j}}^{\infty}}$ defines an element in $Wh_{\chi}(I_{P,\sigma_{\nu}\otimes \bar{\eta}_{j}}^{\infty})$. Our next goal is to define an isomorphism
\[
 \phi:Wh_{\chi}(I_{P,\sigma_{\nu}\otimes\eta}^{\infty})\longrightarrow \bigoplus_{j=0}^{r} W^{j+1}|_{I_{P,\sigma_{\nu}\otimes\eta_{j}}^{\infty}}
\]
such that the following diagram is commutative

\begin{equation}\label{diag2}
 \begin{picture}(250,250)
\put(10,200){$Wh_{\chi}(I_{P,\sigma_{\nu}\otimes\eta}^{\infty})$}
\put(180,200){$\oplus_{j=0}^{r}W^{j+1}|_{I_{P,\sigma_{\nu}\otimes\eta_{j}}^{\infty}}$}
\put(180,140){$\oplus_{j=0}^{r}Wh_{\chi}(I_{P,\sigma_{\nu}\otimes \bar{\eta}_{j}}^{\infty})$}
\put(180,80){$\oplus_{j=0}^{r} V_{\sigma\otimes \bar{\eta}_{j}}'$}
\put(180,20){$(V_{\sigma}\otimes F)'$}
\put(85,203){\vector(1,0){90}}
\put(210,190){\vector(0,-1){35}}
\put(210,130){\vector(0,-1){35}}
\put(210,70){\vector(0,-1){35}}
\put(40,190){\vector(1,-1){155}}
\put(130,210){$\phi$}
\put(120,120){$\Phi_{P,\sigma_{\nu}\otimes\eta}$}
\put(220,110){$\Phi_{P,\sigma_{\nu}\otimes\bar{\eta}_{j}}$}
\put(220,50){$X_{j}/X_{j+1}\cong F_{r-2j}$}
\end{picture}
\end{equation}

If we can find such an isomorphism, then $\Phi_{P,\sigma_{\nu}\otimes\bar{\eta}_{j}}$ would be an isomorphism for all $j$. In particular if $(\eta,F)$ is a representation such that the action of $M$ on $F_{r}$ is trivial, then $\sigma_{\nu}\otimes\bar{\eta}_{r} \cong \sigma_{\nu-rh'}$ and hence $\Phi_{P,\sigma_{\nu-rh'}}$ would be an isomorphism. By iterating this whole argument it is easy to see that this would imply that $\Phi_{P,\sigma_{\nu}}$ is an isomorphism for all $\nu$. We will now proceed to the definition of the isomorphism $\Phi$.

Let $\phi\in I_{P,\sigma_{\nu}\otimes\eta}^{\infty}$, and define $\phi_{j}:G\longrightarrow V_{\sigma}\otimes X_{r-2j}$ by
\[
 \phi_{j}(k)=p_{j}(\phi(k)) \qquad \forall k \in K,
\]
where $p_{j}:V_{\sigma}\otimes F \longrightarrow V_{\sigma}\otimes F_{r-2j}$ is the natural projection. Since $p_{j}$ commutes with the action of $K_{M}$, $\phi_{j}|_{K}\in I_{M\cap K, \sigma_{M \cap K}\otimes\eta_{j}}^{\infty}$. It is clear that $\phi|_{K}=\sum_{j=0}^{r}\phi_{j}|_{K}$. Now given $\lambda \in Wh_{\chi}(I_{P,\sigma_{\nu}\otimes\eta}^{\infty})$, define $\lambda_{j}$ by setting $\lambda_{j}(\phi)=\lambda(\phi_{j})$. Then $\lambda = \sum_{j=0}^{r}\lambda_{j}$, and $\lambda_{j}\in W^{j+1}$ for all $j$. Define
\[
 \phi(\lambda)=(\tilde{\lambda}_{0},\ldots,\tilde{\lambda}_{r}), \qquad \mbox{with $\tilde{\lambda}_{j}=\lambda_{j}|_{I_{P,\sigma_{\nu}\otimes\eta_{j}}^{\infty}}$.}
\]
From this definitions it is clear that $\phi$ is an isomorphism, and besides if $\phi\in U_{P,\sigma_{\nu}\otimes\eta}$, then
\begin{eqnarray}
 \lambda(\phi) & = & \sum_{j=0}^{r}\lambda_{j}(\phi_{j})=\sum_{j=0}^{r}\tilde{\lambda}_{j}(\phi_{j}) \nonumber \\
\mu_{\lambda}(\int_{N}\chi(n)^{-1}\phi(w_{M}n)\, dn ) & = & \sum_{j=0}^{r}\mu_{\tilde{\lambda}_{j}}(\int_{N}\chi(n)^{-1}\phi_{j}(w_{M}n)\, dn ). \label{eq:mu-lambda-tilde}
\end{eqnarray}
Let $v\in V_{\sigma}\otimes F$, and let $v_{j}=p_{j}(v)$. Lets assume that $\phi(w_{M})=v$ and let $\{u^{l}\}_{l}$ be an approximate identity on $N$.
 Set $\phi^{l}(w_{M}n)=u^{l}(n)\chi(n)\phi(w_{M}n)$. Then from equation (\ref{eq:mu-lambda-tilde})

\begin{eqnarray}
 \lim_{l\rightarrow \infty}\mu_{\lambda}(\int_{N}\chi(n)^{-1}\phi^{l}(w_{M}n)\, dn )  & = & \lim_{l\rightarrow \infty} \sum_{j=0}^{r}\mu_{\tilde{\lambda}_{j}}(\int_{N}\chi(n)^{-1}\phi_{j}^{l}(w_{M}n)\, dn )  \nonumber \\
\mu_{\lambda}(\lim_{l\rightarrow \infty}  \int_{N} u^{l}(n)\phi(w_{M}n)\, dn ) & = & \sum_{j=0}^{r}\mu_{\tilde{\lambda}_{j}}( \lim_{l\rightarrow \infty} \int_{N}u^{l}(n)\phi_{j}(w_{M}n)\, dn ) \nonumber \\
\mu_{\lambda}(\phi(w_{M}))& = &\sum_{j=0}^{r}\mu_{\tilde{\lambda}_{j}}(\phi_{j}(w_{M})\nonumber \\
\mu_{\lambda}(v)& = &\sum_{j=0}^{r}\mu_{\tilde{\lambda}_{j}}(v_{j}).
\end{eqnarray}
Which shows that the diagram (\ref{diag2}) is commutative, which implies that  $\Phi_{P,\sigma_{\nu-r}}$ is an isomorphism. Iterating this argument we can show that $\Phi_{P,\sigma,\nu}$ is an isomorphism for all $\nu \in \mathfrak{a}'$ finishing the proof of (i).

We will now prove (ii). Define
\[
 \gamma_{\mu}:\mathfrak{a}'\longrightarrow (I^{\infty}_{K\cap M,\sigma|_{K\cap M}})'
\]
by $\gamma_{\mu}(\nu)(\phi)=\Phi_{P,\sigma,\nu}^{-1}(\mu)(\phi_{P,\sigma,\nu})$. We want to prove that $\gamma_{\mu}$ is weakly holomorphic, i.e., for all $\phi\in I^{\infty}_{K\cap M,\sigma|_{K\cap M}}$, the map $\nu \mapsto \gamma_{\mu}(\nu)(\phi)$ is holomorphic. It follows immediately from lemma \ref{lemma:hyperplane} that if $\operatorname{R}\nu > q_{\sigma}$, then $\gamma_{\mu}(\nu)(\phi)=\mu\circ J_{P,\sigma,\nu}^{\chi}(\phi)$ and hence $\gamma_{\mu}$ is weakly holomorphic in this region. To show that $\gamma_{\mu}$ is weakly holomorphic everywhere, we will make use of the arguments given in the proof of (i) to implement a shift equation for $\gamma_{\mu}$.

Let $\nu\in \mathfrak{a}'$ and $\phi\in I^{\infty}_{K\cap M,\sigma|_{K\cap M}}$ be arbitrary. By definition
\[
 \gamma_{\mu}(\nu-r)(\phi)=\Phi_{P,\sigma,\nu-r}^{-1}(\mu)(\phi_{P,\sigma,\nu-r})=\lambda(\phi_{P,\sigma,\nu-r})
\]
for some $\lambda \in Wh_{\chi}(I_{P,\sigma,\nu-r}^{\infty})$. Let $\{v_{j}\}_{j=1}^{m}$ be a basis of $F$, and let $\{l_{j}\}_{j=1}^{m}$ be its dual basis. Chasing the diagrams (\ref{diag1}) and (\ref{diag2}) we can define $\eta_{j}\in Wh_{\chi}(I_{P,\sigma,\nu}^{\infty})$, $j=1,\ldots,m$, and $h\in I_{P,\sigma_{\nu}\otimes\eta}$ such that
\begin{eqnarray}
 \gamma_{\mu}(\nu-r)(\phi)& = &\lambda(\phi_{P,\sigma,\nu-r})=\check{\Gamma}(\sum \eta_{j}\otimes l_{j})(h) \nonumber\\
                       & = & \Gamma(\sum \eta_{j} \otimes l_{j})(\check{h}) \nonumber \\
                       & = & (\sum \eta_{j} \otimes l_{j})(\Gamma^{T}\check{h}). \label{eq:nu-r}
\end{eqnarray}
Now since $\Gamma^{T}\check{h} \in I_{P,\sigma,\nu}^{\infty}\otimes F$ we can find $\phi_{j}\in I^{\infty}_{K\cap M,\sigma|_{K\cap M}}$, $j=1,\ldots,m$, such that
\[
 \Gamma^{T}\check{h}=\sum (\phi_{j})_{P,\sigma,\nu} \otimes v_{j}.
\]
Hence
\begin{equation}\label{eq:nu}
 \check{\Gamma}(\sum \lambda_{j}\otimes l_{j})(h)= \sum \eta_{j}((\phi_{j})_{P,\sigma,\nu})=\sum \mu_{\eta_{j}}\circ J_{P,\sigma,\nu}^{\chi}(\phi_{j})=\sum \gamma_{\eta_{j}}(\nu)(\phi_{j}).
\end{equation}
Then by (\ref{eq:nu-r}) and (\ref{eq:nu})
\begin{equation}\label{eq:functionaleq}
 \gamma_{\mu}(\nu-r)(\phi) = \sum \gamma_{\eta_{j}}(\nu)(\phi_{j}).
\end{equation}
This is the desired shift equation which shows that $\gamma_{\mu}$ is weakly holomorphic everywhere.
\end{proof}

\medskip
\section{Bessel models for admissible representations.}

In this section we derive an estimate on the dimension of the space
$Wh_{\tau,\chi}(V)$ for $(\pi,V)$ an admissible finitely generated smooth
Fr\'{e}chet representation of $G$ as a corollary to Corollary 13.\medskip\ Let
$V_{K}$ deonote the underlying $(\mathfrak{g},K)$ module of $K$ finite vectors
in $V$. Let $\widetilde{V_{K}}$ be the admissible dual module of $V_{K}$ and
let $\widetilde{W_{K}}$ be an irreducible $(Lie(P),K_{M})$ ($K_{M}=K\cap M$)
quotient of $\widetilde{V_{K}}/\mathfrak{n}\widetilde{V_{K}}$. Then Frobenius
reciprocity implies that there is an injective homorphism of $\widetilde{V_{K}%
}$ into the $K$-finite induced module from $P$ to $G$ of $\widetilde{W_{K}}$.
Let $W_{K}$ denote the $K\cap M$ finite dual module to $\widetilde{W_{K}}$.
Then we have a homomorphism of the $K$-finite induced representation from $P$
to $G$ of $\delta_{P}\otimes W_{K}$ onto $V_{K}$. Let $(\sigma,H_{\sigma})$
denote the canonical completion of $\delta_{P}^{\frac{1}{2}}\otimes W_{K}$ as
a smooth Fr\'{e}chet module of moderate growth. Then the Casselman-Wallach
theorem implies that there exists a continuous surjective $G$-homorphism, $T$,
of $I_{P,\sigma}^{\infty}$ onto $V$. Let $T^{\prime}$ be the adjoint of $T$
mapping $V^{\prime}$ into $\left(  I_{P,\sigma}^{\infty}\right)  ^{\prime}$
since $T$ is surjective, $T^{\prime}$ is injective. With this notation in hand
Corollary 13 implies

\begin{proposition}
Let $V$ be an irreducible, smooth. Fr\'{e}chet module of moderate growth for
$G$. Let $(\tau,F_{\tau})$ be an irreducible represenation of $K_{M}$, let
$\chi$ be a unitary character of $N$ such that $M_{\chi}=K_{M}$ then
\[
\dim Wh_{\chi,\tau}(V)\leq\dim Hom_{K_{M}}(W_{K},F_{\tau})
\]
where $W_{K}$ is the $K_{M}$ finite dual module of any irreducible
$(Lie(M),K_{M})$ quotient of $V_{K}/\mathfrak{n}V_{K}$.
\end{proposition}

This result expresses an interesting relationship between Bessel models and
conical models (i.e. imbeddings into principal series).

\medskip
\begin{center}
{\Large References}
\end{center}

\medskip
\noindent\lbrack JSZ] D. Jiang, B. Sun, C, Zhu, Uniqueness of Bessel models:
the Archimedean Case, To Appear, Geometric and Functional Analysis.

\noindent\lbrack KV] J.A.C. Kolk, V.S. Varadarajan, On the transverse symbol
of vectorial distributions and some applications to harmonic analysis, Indag.
Math., 7 (1996), 67-96.

\noindent\lbrack L] E. Lynch, Generalized Whittaker vectors and representation
theory, Thesis, M.I.T., 1979.

\noindent\lbrack S] C.L. Siegel, Symplectic Geometry, Academic Press, New
York, 1964.

\noindent\lbrack W1] Nolan R. Wallach, Lie algebra cohomology and holomorphic
continuation of generalized Jacquet integrals, Representations of Lie Groups
Kyoto, Hiroshima, 1986,123-151, Adv. Stud. Pure Math., 14, Academic Press,
Boston, MA 1988.

\noindent\lbrack W2] Nolan R. Wallach, Holomorphic continuation of generalized
Jacquet integrals for degenerate principal series, Represent. Theory, 10,
(2006), 380-398.

\noindent\lbrack W3] Nolan R. Wallach, Real Reductive Groups 1, Academic
Press, Boston, 1988.

\noindent\lbrack W4] Nolan R. Wallach, Real Reductive Groups 2, Academic
Press, Boston, 1992.

\noindent\lbrack W5] Nolan R. Wallach, Generalized Whittaker models for
holomomprhic and quaternionic representations, Comment. Math. Helv. 78 (2003), 266-307.

\end{document}